\documentclass[]{article}
\usepackage{amsfonts}
\usepackage{amssymb}
\usepackage[letterpaper, left=2.5cm, right=2.5cm, top=2.5cm,
bottom=2.5cm,dvips]{geometry}
\usepackage{verbatim}

\usepackage[T1]{fontenc}
\usepackage{graphicx}
\usepackage{diagbox}
\usepackage{xcolor}
\usepackage{amsmath}
\usepackage{array,multirow}
\usepackage{dsfont}
\usepackage{amsthm}
\usepackage{enumerate}

\usepackage[pagebackref]{hyperref}
\hypersetup{
	colorlinks,
	linkcolor={blue!60!black},
	citecolor={green!60!black},
	urlcolor={green!60!black}
}

\newtheorem{lemma}{Lemma}[section]
\newtheorem{corollary}[lemma]{Corollary}

\newtheorem{theorem}{Theorem}[section]
\newenvironment{thmbis}[1]
 {\renewcommand{\thetheorem}{\ref{#1}}%
 \addtocounter{theorem}{-1}%
 \begin{theorem}}
 {\end{theorem}}

\newtheorem{claim}{Claim}

\newcommand{\I}{\mathcal {I}}

\newcommand{\N}{\mathcal {N}}

\newcommand{\R}{\mathcal {R}}

\newcommand{\bbZ}{{\mathbb Z}}

\newcommand{\bbC}{{\mathbb C}}

\newcommand{\va}{{\mathbf{a}}}

\newcommand{\vo}{{\mathbf{o}}}

\newcommand{\vx}{{\mathbf{x}}}
\newcommand{\vy}{{\mathbf{y}}}

\newcommand{\Rmnum}[1]{\uppercase\expandafter{\romannumeral #1}}

\numberwithin{equation}{section}

\usepackage{xcolor}

\begin{document}
\title{Trace reconstruction of matrices and hypermatrices}
\date{}
\author{Wenjie~Zhong\thanks{W. Zhong ({\tt zhongwj@mail.ustc.edu.cn}) is with the School of Mathematical Sciences, University of Science and Technology of China, Hefei, 230026, Anhui, China.} ~and~Xiande~Zhang
\thanks{X. Zhang ({\tt drzhangx@ustc.edu.cn}) is with the School of Mathematical Sciences,
University of Science and Technology of China, Hefei, 230026, and with Hefei National Laboratory, University of Science and Technology of China, Hefei 230088, China.
}
}



\maketitle

\begin{abstract}
A \emph{trace} of a sequence is generated by deleting each bit of the sequence independently with a fixed probability. The well-studied \emph{trace reconstruction} problem asks how many traces are required to reconstruct an unknown binary sequence with high probability. In this paper, we study the multivariate version of this problem  for  matrices and hypermatrices, where a trace is generated by deleting each row/column of the matrix or each slice of the hypermatrix independently
with a constant probability. Previously, Krishnamurthy et al. showed that $\exp(\widetilde{O}(n^{d/(d+2)}))$ traces suffice to reconstruct any unknown $n\times n$ matrix (for $d=2$) and any unknown $n^{\times d}$ hypermatrix. By developing a dimension reduction procedure and establishing a multivariate version of the Littlewood-type result, we improve this upper bound by showing that $\exp(\widetilde{O}(n^{3/7}))$ traces suffice to reconstruct any unknown $n\times n$ matrix, and $\exp(\widetilde{O}(n^{3/5}))$ traces suffice to reconstruct any unknown $n^{\times d}$ hypermatrix. This breaks the tendency to trivial $\exp(O(n))$  as the dimension $d$ grows.
\end{abstract}

\section{Introduction}

Reconstructing a combinatorial object from a limited amount of sub-information is a fundamental problem in computer science and has been widely studied, for example, the sequence reconstruction problem initiated by Levenshtein in 1997 \cite{levenshtein1997reconstruction,levenshtein2001efficient,levenshtein2001efficientfrom,konstantinova2007reconstruction,gabrys2018sequence, ye2023reconstruction}, the deck problem introduced by Kalashnik in 1973 \cite{kalashnik1973reconstruction,krasikov1997reconstruction, foster2000improvement,dudik2003reconstruction,kos2009reconstruction,zhang2024reconstruction}, and the graph reconstruction problem with the famous Ulam's conjecture \cite{o1970ulam,bondy1977graph,bollobas1990almost,lauri2016topics,kostochka2020reconstruction,groenland2021size}.

Motivated by the study of reconstructing DNA sequences, the trace reconstruction problem is a probabilistic version of the sequence reconstruction  problem, which was first introduced by Batu et al. \cite{batu2004reconstructing}. The goal of this problem is to determine the minimum number of \emph{traces} $T(n)$ that are required to reconstruct an unknown string $x\in \{0,1\}^n$ {with high probability}. Here, a {trace} of $x$ is a random substring of $x$ obtained by deleting each bit of $x$ with a fixed probability $q$ independently. The term \emph{with high probability} means that the probability can be arbitrarily close to $1$ as the string length $n$ grows.

For the upper bound of $T(n)$, Holenstein et al. \cite{holenstein2008trace} first showed that $\exp(\sqrt{n}\text{poly}(\log n))$ traces suffice. Then Nazarov and Peres \cite{nazarov2017trace} and De et al. \cite{de2017optimal} independently proved that $T(n)\leq \exp(O(n^{1/3}))$, which was recently improved to $\exp(O(n^{1/5}\log^5 n))$ by Chase \cite{chase2021separating}. As for the lower bound, Batu et. al. \cite{batu2004reconstructing} first proved that $\Omega(n)$ traces are necessary. Then Holden and Lyons \cite{holden2020lower} later proved that $T(n)\geq \Omega(n^{5/4}/\sqrt{\log n})$, which was recently improved to $ \Omega(n^{3/2}/\log^7 n))$ by Chase \cite{chase2021new}. More works on the trace reconstruction problem can be found in \cite{peres2017average,holden2020subpolynomial,ban2019beyond,ban2019efficient,cheraghchi2020coded,brakensiek2020coded,narayanan2020population,
narayanan2021circular,davies2021reconstructing,krishnamurthy2021trace,davies2021approximate,chase2021approximate}.

For a generalization of the trace reconstruction problem to  dimension two,  Krishnamurthy et al. \cite{krishnamurthy2021trace} considered this problem for matrices. Here a trace of a matrix $X\in \{0,1\}^{n\times n}$ is a random submatrix of $X$ by deleting each row and each column independently with probability $q$. They showed that  $\exp(\widetilde{O}(n^{1/2}))$\footnote{For functions $f$ and $g$, we say $f=\widetilde{O}(g)$
if $|f|\leq C|g|\log^C|g|$
for some constant $C$.} traces suffice to recover an arbitrary matrix. This problem can be further extended to general higher dimensions, say dimension $d$, for which a trace is defined to be a random sub-hypermatrix obtained by deleting each slice, i.e., a $(d-1)$-dimensional sub-hypermatrix of full size, independently with a constant probability $q$.
Krishnamurthy et al. \cite{krishnamurthy2021trace} similarly proved that $\exp(\widetilde{O}(n^{d/(d+2)}))$ traces suffice to recover an arbitrary $d$-dimensional hypermatrix $X\in \{0,1\}^{n^{\times d}}$. Note that this bound gets closer to trivial when the dimension $d$ becomes big.

In this paper, we aim to improve the upper bounds given by Krishnamurthy et al. \cite{krishnamurthy2021trace} for general dimensions.  We show that $\exp(\widetilde{O}(n^{3/7}))$ traces are enough for trace reconstruction of a matrix, $\exp(\widetilde{O}(n^{5/9}))$ traces are enough for trace reconstruction of a $3$-dimensional hypermatrix, say a cube. For general dimension $d$, we further improve the upper bound $\exp(\widetilde{O}(n^{d/(d+2)}))$ to $\exp(\widetilde{O}(n^{3/5}))$, where the main exponential term $n^{3/5}$ is independent of $d$, breaking the asymptotic tendency to trivial $n$ as $d$ grows. Our results are stated as follows.

\begin{theorem}
For any fixed deletion probability $q\in(0,1)$, the following hold with high probability:
\begin{itemize}
  \item[(1)]\label{theorem-1} $\exp(O(n^{3/7}\log^{10/3}n))$ traces suffice to reconstruct an arbitrary matrix $X\in \{0,1\}^{n\times n}$;
  \item[(2)]\label{theorem-2} $\exp(O(n^{5/9}\log^{5/2}n))$ traces suffice to reconstruct an arbitrary $3$-dimensional hypermatrix $X\in \{0,1\}^{n\times n\times n}$; and
  \item[(3)]\label{theorem-3} for any fixed dimension $d\geq 4$, $\exp(O(n^{3/5}\log n))$ traces suffice to reconstruct an arbitrary hypermatrix $X\in \{0,1\}^{n^{\times d}}$.
\end{itemize}

\end{theorem}

To get the upper bound $\exp(\widetilde{O}(n^{3/5}))$ in Theorem~\ref{theorem-1} (3), we generalize the Littlewood-type result for single-variable polynomials in \cite{borwein1997littlewood} to multi-variable polynomials in the following theorem, which we think is interesting itself.
We actually concern about a stronger version with sparse property, which was established by Chase \cite[Theorem 6.2]{chase2021separating} for single-variable polynomials.
Let $h(z_1,\ldots, z_d)= \sum c_{k_1\cdots k_d}z_1^{k_1}\cdots z_d^{k_d}$ be a $d$-variable complex polynomial with each coefficient $c_{k_1\cdots k_d}\in \{0,\pm 1\}$. We say $h$ is \emph{$s$-sparse}, if there is some $1\leq j\leq d$ such that $|k_j-k'_j|\ge s$ whenever two coefficients $c_{k_1\cdots k_d}=c_{k'_1\cdots k'_d}\neq 0$.


\begin{theorem}\label{thm-sparsepolynomial-multivariate-lowerbound}
Let $h(z_1,\ldots, z_d)= \sum_{k\in [n]^d} c_kz_1^{k_1}\cdots z_d^{k_d}$ be a nonzero $n^\mu$-sparse polynomial with $\mu\in [0,1)$ and each $c_k\in \{0,\pm 1\}$. Then for any $\Delta\ge 1$ and any $L$ with $1\le L\le n^\Delta$,
\[\max_{z_1,\ldots, z_d\in \gamma(L)}|h|\ge e^{-O(\Delta L n^{1-\mu} \log n)}.\]
Here $[n]=\{0,1,\ldots,n-1\}$ and $\gamma(L):= \{e^{i\theta}:-\pi/L\leq \theta\leq \pi/L\}$ for any $L\ge 1$.\end{theorem}

Note that when $\mu=0$, the sparse condition in Theorem~\ref{thm-sparsepolynomial-multivariate-lowerbound} is trivial, where we can get a lightly stronger result \[\max_{z_1,\ldots, z_d\in \gamma(L)}|h|\ge e^{-O(Ln \log n)},\] for any $L\geq 1$. When $L\geq C n^{\frac{1}{d-1}}$ for a large constant  $C$, this greatly improves the lower bound $e^{-O(L^d \log n)}$ in \cite[Lemma 22]{krishnamurthy2021trace} by Krishnamurthy et al.

%

We further remark that as established by Chase \cite{chase2021separating}, there are  intriguing similarities between trace reconstruction and the $k$-deck problem\footnote{The $k$-deck problem for sequences is to find the minimum value of $k$ such that
we can reconstruct any binary sequence of length $n$ from the multi-set of
all subsequences of length $k$.}, including their proof methods. However, the best known upper bound on $k$ for the $d$-dimensional $k$-deck problem is $k=O(n^\frac{d}{d+1})$ \cite{zhang2024reconstruction}, which is close to trivial when $d$ grows. The problem whether an absolute constant $\alpha\in (0,1)$ exists such that  $k=O(n^\alpha)$ for all $d$ was raised in \cite{zhang2024reconstruction}. We wonder that whether our method for trace reconstruction is applicable to improve the upper bound to the  $k$-deck problem.

Our paper is organized as follows. In Section \ref{section2}, we first recall  the proof idea  of Nazarov and Peres \cite{nazarov2017trace} and Chase \cite{chase2021separating} for trace reconstruction of sequences, then we set up the general framework for trace reconstruction of (hyper)matrices, including the concepts of multidimensional generating function (identity), and the key tools for our trace reconstruction.  In Section 3, we design a dimension reduction procedure, which helps to analyse the recursive structure and sparsity of certain generating functions. In Section 4, we improve the previous upper bounds for trace reconstruction of matrices and cubes, that is Theorem~\ref{theorem-1} (1)-(2). In Section 5, we prove Theorem~\ref{thm-sparsepolynomial-multivariate-lowerbound} and Theorem~\ref{theorem-1} (3) for $d\ge 4$.

%

\section{Preliminary}\label{section2}

In this section, we first briefly review the  proof framework of Nazarov and Peres \cite{nazarov2017trace} and Chase \cite{chase2021separating} for trace reconstruction of sequences. Then based on their ideas, we set up necessary notations and the general framework for trace reconstruction of (hyper)matrices.

For any integer $l$, let $[l]:=\{0,1,\ldots,l-1\}$ and $[-l,l]:=\{-l,-(l-1),\ldots,-1,0,1,\ldots,l\}$. For a number $a$ and a set $S$ of numbers, let $a+S:=\{a+s:s\in S\}$ and $aS=\{a\cdot s:s\in S\}$. Let $\gamma:= \{e^{i\theta}:-\pi\leq \theta\leq \pi\}$ be the unit circle in the complex plane, and $\gamma(L):= \{e^{i\theta}:-\pi/L\leq \theta\leq \pi/L\}$ for any $L\ge 1$. Numbers $p,q\in (0,1)$ are always fixed and satisfy $p=1-q$. By\cite[Eq.~(2.3)]{nazarov2017trace}, for any $z\in \gamma(L)$, the norm of $w=\frac{z-q}{p}$ always satisfies
\begin{align}\label{eq-estimation-w}
|w|\le \exp(O(L^{-2})). 
\end{align}

\subsection{Trace reconstruction of sequences}

To obtain the upper bound $\exp(O(n^{1/3}))$, Nazarov and Peres \cite{nazarov2017trace} used \emph{single-bit} statistics to distinguish between any two distinct sequences $x,y\in \{0,1\}^n$.  The main idea is to find a certain index $j$ such that the number of traces with a $1$ at this position will differ
substantially enough in expectation for $x$ and $y$. To establish the existence of such a $j$,  they used the following simple but useful relation between the entries of traces and the entries of the original sequence, 
\begin{equation}\label{generating-identity-sequence}
  \begin{aligned}
\mathbb{E}[\sum_{j\geq 0}\widetilde{X}_jw^j]=p\sum_{k\geq 0}^{n-1}x_kz^k,
\end{aligned}
\end{equation}
which is valid for any $w,z\in \mathbb{C}$ satisfying $z=pw+q$.
Here $x=(x_0,\ldots,x_{n-1})\in \{0,1\}^n$ is the original sequence, $\widetilde{X}$ is a random trace of $x$ (pad $\widetilde{X}$ with zeros to the right) through the deletion channel with probability $q$.
Applying Eq. (\ref{generating-identity-sequence}) with $a=(a_0,\ldots,a_{n-1}):=x-y$, one obtains
\begin{equation}\label{genidentity-seq}
  \begin{aligned}
\mathbb{E}[\sum_{j\geq 0}(\widetilde{X}_j-\widetilde{Y}_j)w^j]=p\sum_{k=0}^{n-1}a_kz^k\triangleq p A(z).
\end{aligned}
\end{equation}
Then it suffices to show that the polynomial $A(z)$
is not too small for some $z$ close to $1$ on the unit circle, which will deduce the existence of $j$ satisfying $|\mathbb{E}[\widetilde{X}_j-\widetilde{Y}_j]|\geq \exp(-O(n^{1/3}))$ by pigeonhole principle.

\vspace{0.5cm}
To further improve the upper bound $\exp(O(n^{1/3}))$ to $\exp(\widetilde{O}(n^{1/5}))$, Chase \cite{chase2021separating} used the \emph{multiple-bit} statistics to distinguish sequences $x,y\in \{0,1\}^n$ when they coincide on a long prefix, that is, $y$ looks like $x$ from the beginning. Note that this is the hard case to distinguish $x$ and $y$.

The idea of \emph{multiple-bit} statistics is to find a set of $l$ indices, say $0\leq j_0<j_1<\cdots<j_{l-1}\leq n-1$, such that the number of traces with a certain subsequence, say $W\in \{0,1\}^l$, on these indices will differ
substantially enough in expectation for $x$ and $y$. To find such a set of indices,
 Chase generalized Eq. (\ref{genidentity-seq}) to the following similar but much more complicated identity\footnote{In this paper, we abuse the notations to make formulas tidy, whose meaning can be got from the content. },
\begin{equation}\label{idofchase}
  \begin{aligned}\mathbb{E}[\sum_j(1_{\widetilde{X}_j=W}-1_{\widetilde{Y}_j=W})w^{\odot j}]=p^l\sum_k(1_{x_k=W}-1_{y_k=W})z^{\odot k}\triangleq p^l B(z),
  \end{aligned}
\end{equation}
which is valid for any $w=(w_0,\ldots,w_{l-1}), z=(z_0,\ldots,z_{l-1})\in \mathbb{C}^l$ satisfying  $z_i=pw_i+q$, $0\leq i\leq  l-1$. Here, the two indices $j=(j_0,j_1,\ldots,j_{l-1})$ and $k=(k_0,k_1,\ldots, k_{l-1})$ of the two sums run over
\[\I(n,l):= \{(j_0,j_1,\ldots, j_{l-1}) \in \bbZ ^l: 0\le j_0< j_1< \cdots < j_{l-1}\le n-1 \},\]
which is the collection of all sets of $l$ positions among the $n$ positions. Then $x_k:= (x_{k_0},x_{k_1}, \ldots, x_{k_{l-1}})$, and $z^{\odot k}$ is defined as
\begin{align*}
z^{\odot k}:= z_0^{k_0}z_1^{k_1-k_0-1} \cdots z_i^{k_i-k_{i-1}-1} \cdots z_{l-1}^{k_{l-1}-k_{l-2}-1}.
\end{align*}
 When $l=1$ and $W=1$, Eq. (\ref{idofchase}) reduces to Eq. (\ref{genidentity-seq}).

 As before, one needs to show that $B(z)$ in Eq. (\ref{idofchase}) is not too small on certain point, which is a multi-variable power series. Note that when $k$ is a set of consecutive positions, for example, $k=k_0+[l]$, then $z^{\odot k}=z_0^{k_0}$ is single-variable. The following corollary from Chase \cite{chase2021separating} tells that the partial sum over such consecutive positions is enough to control the total sum.


\begin{lemma}[Corollary 6.4, Chase \cite{chase2021separating}]\label{lem-sequence-contiguous-function}
For any distinct $x,y\in \{0,1\}^n$ and any $p\in (0,1)$, if for some $W\in \{0,1\}^l$ with $l\le n$, some $z_0\in \gamma$ and some constant $C> 0$, we have
\begin{equation}\label{consebound}
  |B'(z_0)|:=|\sum_{k\ge 0}(1_{x_{k+[l]}=W}-1_{y_{k+[l]}=W})z_0^k|\ge \exp(-Cl \log^5 n),
\end{equation}
then there are $z_1,\ldots, z_{l-1}\in [1-2p, 1]$ and a constant $C'> 0$ such that
\[|\sum_{k\in \I(n,l)}(1_{x_k=W}-1_{y_k=W})z^{\odot k}|\ge \exp(-C'l \log^5 n),\]
where $z=(z_0,z_1,\ldots, z_{l-1})$. 
\end{lemma}

Because $x$ and $y$ coincide on a long prefix, the polynomial $B'(z_0)$ in Eq. (\ref{consebound}) could be very sparse for some carefully chosen $W$, i.e., a lot of zero coefficients for lower degrees.
Combining with the following result with $l= 2n^{1/5}, \mu=1/5$ and some complex analysis, Chase \cite{chase2021separating} gave a better lower bound $|\mathbb{E}[1_{\widetilde{X}_j=W}-1_{\widetilde{Y}_j=W}]|\ge \exp(-\widetilde{O}(n^{1/5}))$ for some $j$, and then $\exp(\widetilde{O}(n^{1/5}))$ traces suffice.

\begin{theorem}[Theorem 6.2, Chase \cite{chase2021separating}]\label{thm-sparsepolynomial-lowerbound}
For any $\mu\in (0,1)$, 
there is some constant $C> 0$ such that for any $n\ge 2$ and any polynomial $f(z)= 1- \epsilon z^d+ \sum_{j=n^\mu}^n c_jz^j$ with $1\le d< n^\mu, \epsilon\in \{0,1\}$ and $|c_j|\le 1$ for each $j$, it holds that
\[\max_{|\theta|\le n^{-2\mu}}|f(e^{i\theta})|\ge \exp(-Cn^\mu \log^5n).\]
\end{theorem}



\subsection{Set up for higher dimensions}

In this subsection, we set up necessary notations and preliminary results for trace reconstruction of (hyper)matrices. For a fixed integer $d\geq 2$, we write $\{0,1\}^{n^{\times d}}$ instead of $\{0,1\}^{n\times \cdots \times n}$ for convenience, each element of which is an $n^{\times d}$-hypermatrix. We call $n$ the length and $d$ the dimension.

For $X\in \{0,1\}^{n^{\times d}}$, the position of every entry is indexed by an element $k=(k_1,k_2,\ldots,k_d)\in [n]^d$, and this entry is denoted by $X_{k}$.
For a  sub-hypermatrix of $X$  with length $l$ projected to the first $r$ dimensions, we can also define its position by a single notation, still using $k=(k_1,k_2,\ldots, k_d)$, where $k_i= (k_{i0},k_{i1},\ldots, k_{i,l-1})\in \I(n,l)$ for $1\le i\le r$ and $k_i\in [n]$ for $r+1\le i\le d$.
Then the sub-hypermatrix is also denoted by $X_k$.
Let $\I(n^{\times d},l^{\times r})$ be the collection of all such positions $k$ for fixed  $0< l\le n$ and $ 0\le r\le d$.
%
For example, when $X\in \{0,1\}^{n^{\times 2}}$ and $k=((1,4,5),2)$, $X_k=(X_{(1,2)},X_{(4,2)},X_{(5,2)})$. When $X_k$ is a block, then each of the $r$ former $k_i$'s for $1\le i\le r$ must be of the form $k_{i0}+[l]$. 
In this case, we simply denote the position by $k+[l]^r$ with $k=(k_1,k_2,\ldots, k_d)\in [n]^d$ indicating the first corner entry's position. Note that when $l=1$ or $r=0$, $\I(n^{\times d},1)=[n]^d$.

Similarly for $0\leq r \leq d$, let $\bbC(d,l^{\times r})=\bbC^l\times\cdots  \times\bbC^l \times \bbC \times \cdots \times\bbC $, where $\bbC^l$ repeats $r$ times and $\bbC$ repeats $d-r$ times.  So any $(z_1,\ldots,z_d)\in \bbC(d,l^{\times r})$ satisfies $z_i=(z_{i0},\ldots, z_{i,l-1})\in \bbC^l$ for $1\le i\le r$ and $z_i\in \bbC$ for $r+1\le i\le d$. When $l=1$ or $r=0$, $\bbC(d,1)=\bbC^d$.

Now, for a $W\in \{0,1\}^{l^{\times r}}$ and a  hypermatrix $X\in \{0,1\}^{n^{\times d}}$, we  define the \emph{$W$-generating function of $X$} as
\[g(z)=\sum_k 1_{X_k=W}z_1^{\odot k_1}\cdots z_{r}^{\odot k_{r}}\cdot z_{r+1}^{k_{r+1}}\cdots z_d^{k_d},\]
where $k= (k_1,\ldots, k_d)$ runs through $\I(n^{\times d},l^{\times r})$, and $z=(z_1,\ldots,z_d)\in \bbC(d,l^{\times r})$.
In fact, this is a polynomial over $\bbC$ with $rl+ d-r$ variables.
When $d=1$, $r=0$ and $W=1$, it reduces to $A(z)$ in Eq. (\ref{genidentity-seq}).
 When $d=r=1$, it reduces to $B(z)$ in Eq. (\ref{idofchase}).  Note that when $r< d$, this function is determined by all $l^{\times r}$-subhypermatrices inside the first $r$ dimensions of $X$. For instance, when $d=2$ and $r=1$, $g$ is the generating function of columns in the matrix $X$.

Accordingly, we have the following  \emph{$W$-generating identity of $X$}, for any $W\in \{0,1\}^{l^{\times r}}$ and $X\in \{0,1\}^{n^{\times d}}$, which generalizes Eq. (\ref{genidentity-seq})
and  Eq. (\ref{idofchase}),
\begin{equation}\label{wgeriden}
  \mathbb{E}[\sum_j 1_{\widetilde{X}_j=W} w_1^{\odot j_1}\cdots w_{r}^{\odot j_{r}}\cdot w_{r+1}^{j_{r+1}}\cdots w_d^{j_d}]=p^{rl +d-r}g(z).
\end{equation}
Here $j= (j_1,\ldots, j_d)$ runs through $\I(n^{\times d},l^{\times r})$, and $w=(w_1,\ldots,w_d)\in \bbC(d,l^{\times r})$ satisfies $z=pw+q$, i.e., each pair of corresponding entries in $z$ and $w$ satisfies this relation. From now on, all appeared $w$ accompanied by $z$ satisfy this relation.
We move the proof  of Eq. (\ref{wgeriden}) to Appendix.

To state the general framework of trace reconstruction of hypermatrices, we define a function $\ell(w)=\max\{|w|,1\}$ for all $w\in \bbC$. Then for simplicity, no matter $w\in \bbC$, or in $\bbC^l$, or in $\bbC(d,l^{\times r})$, we define $\ell(w)$ to be the product of $\ell$-values of all complex numbers in $w$.
\begin{lemma}\label{lem-reduceto-Wgen-function} For any deletion probability $q\in (0,1)$, and
 any hypermatrix $X\in \{0,1\}^{n^{\times d}}$ through the deletion channel, let $m= \Omega(\log n)$ be a positive integer. If for any other $Y\in \{0,1\}^{n^{\times d}}$, there is some  $W\in \{0,1\}^{l^{\times r}}$ with $l=O(m/\log n)$ and $ 0\le r\le d$ such that, the $W$-generating function\footnote{The $W$-generating function of $X-Y$ is the one of $X$ minus the one of $Y$.
} $g$ of $X-Y$ satisfies 
\[|g(z)|\ge \exp(-Cm) \text{ for some $w\in \bbC(d,l^{\times r})$ with } \ell(w)\le \exp({Cm}/{n}),\] for some  constant $C> 0$,
then $\exp(O(m))$ traces suffice to reconstruct $X$ with high probability.
\end{lemma}
\begin{proof} By Eq. (\ref{wgeriden}), we
write the $W$-generating identity of $X-Y$:
\[\mathbb{E}[\sum_j (1_{\widetilde{X}_j=W}-1_{\widetilde{Y}_j=W}) w_1^{\odot j_1}\cdots w_{r}^{\odot j_{r}}\cdot w_{r+1}^{j_{r+1}}\cdots w_d^{j_d}]=p^{rl+ d-r}
g(z).\]
For the right hand side, $|p^{rl+ d-r} g(z)|\ge \exp(-O(l+m))\ge \exp(-O(m))$. Then by the identity,
\begin{align*}
\exp(-O(m))
&\le |\mathbb{E}[\sum_j (1_{\widetilde{X}_j=W}-1_{\widetilde{Y}_j=W}) w_1^{\odot j_1}\cdots w_{r}^{\odot j_{r}}\cdot w_{r+1}^{j_{r+1}}\cdots w_d^{j_d}]|\\
&\le \sum_j |\mathbb{E}[1_{\widetilde{X}_j=W}-1_{\widetilde{Y}_j=W}]|\ell(w)^n\\
&\le \sum_j |\mathbb{E}[1_{\widetilde{X}_j=W}-1_{\widetilde{Y}_j=W}]|\exp(Cm).
\end{align*}
By Pigeonhole Principle, we can find some $j\in \I(n^{\times d},l^{\times r})$ and some constant $C_1> 0$ such that
\[|\mathbb{E}[1_{\widetilde{X}_j=W}-1_{\widetilde{Y}_j=W}]|\ge  \exp(-C_1m).\]

Suppose that $T$ i.i.d. traces $\widetilde{U}^1,\ldots,\widetilde{U}^T$ of the original hypermatrix $X$ are observed. We establish a simple algorithm that on input $T$ traces and a pair $(X,Y)$, outputs $X$ if
\[|\frac{1}{T}\sum_{t=1}^{T}\widetilde{U}_j^t-\mathbb{E}_X[\widetilde{X}_j]|<|\frac{1}{T}\sum_{t=1}^{T}\widetilde{U}_j^t-\mathbb{E}_Y[\widetilde{Y}_j]|\]
for the index set $j$ fixed above. By the union bound and Hoeffding's inequality,

\[\mathbb{P}_X[X~\text{cannot be recovered}]\le \sum_{Y\neq X}\mathbb{P}_X[Y~\text{is output from}~(X,Y)]\le 2^{n^d}\exp(-\frac{T}{2}\exp(-2C_1m)),\]
Taking $T= \exp(C_2m)$ with $C_2> 2C_1+ \frac{d\log n}{m}$ makes the above probability tend to $0$.
\end{proof}
\vspace{0.3cm}

Lemma~\ref{lem-reduceto-Wgen-function} reduces the trace reconstruction problem to the existence of a smaller hypermatrix $W$ such that the $W$-generating function of $X-Y$ has a big value at some restricted point.  Next, we give a useful result about lower bounding a complex polynomial, which will be repeatedly used to derive lower bounds of $W$-generating functions.

\begin{lemma}\label{lem-polynomial-lowerbound}
Let $m$ be a sufficiently large integer and $a=a(m)\in (0,1]$. For a polynomial $f(z)= z^{m_0}(c_0+c_1z+ \cdots)\in \bbC[z]$ with  $m_0\le m$, if $|c_0|\ge e^{-O(m^a)}$ and $\sum_i|c_i|\le e^{O(m^a)}$, then there exist some $w\in \bbC$ with $|w|\le 1$ and some constant $C> 0$, such that
\[|f(z)|\ge \exp(-Cm^{\frac{1+2a}{3}}), \text{ where $z=pw+q$}.\]
In particular, when $|c_0|=1$ and $|c_i|\le 1$, there exists some $w$ with $|w|\le 1$, such that $|f(z)|\ge \exp(-O(m^{1/3}))$. 
\end{lemma}

\begin{proof}
Denote $f(z)= z^{m_0} f_1(z)$. Then $|f_1(z)|\le \sum_i|c_i|\le e^{O(m^a)}$ for any $z$ with $|z|\le 1$. For some positive integer $L$, define a polynomial
\[F(z)= \prod_{1\le a\le L} f_1(ze^{2\pi i\frac{a}{L}}).\]
By the maximum modulus principle, there exists some $z'\in \rho\gamma$ with $\rho= 1-\frac{7}{pL^2}$ such that
\[|F(z')|\ge |F(0)|= |f_1(0)|^L= |c_0|^L\ge e^{-O(m^a L)}.\]
We may choose $a$ in $\{1,2,\ldots,L\}$ such that $z:=z'e^{2\pi i\frac{a}{L}}\in \rho\gamma(L)$. Fixing this $z$, we have
\[e^{-O(m^a L)}\le |F(z')|\le |f_1(z)|\cdot e^{O(m^a (L-1))}.\]
Then $|f_1(z)|\ge e^{-O(m^a L)}$. Taking $L= \lceil\frac{4}{p}m^{\frac{1-a}{3}}\rceil$, we have
\[|f(z)|= \rho^{m_0}|f_1(z)|\ge (1-\frac{7}{pL^2})^{m}\cdot \exp(-O(m^a L))\ge \exp(-Cm^{\frac{1+2a}{3}})\]
for some constant $C> 0$.

It is left to show that the corresponding $w=\frac{z-q}{p}$ satisfies $|w|\le 1$, which is true by the following fact:
\begin{align}\label{eq-estimation-w-2}
|w|\le 1, \text{~~ whenever $z\in \rho \gamma(L)$ with $L\ge \frac{4}{p}$ and $\rho= 1-\frac{7}{pL^2}$}.
\end{align}
Indeed, since $z=\rho e^{i\theta}$ with $|\theta|\le \pi/L$, we have
\begin{align*}
|w|^2&=|1+\frac{\rho e^{i\theta}-1}{p}|^2\leq 1+\frac{2}{p}(\rho-1)+\frac{1}{p^2}|\rho e^{i\theta}-1|^2= 1+\frac{2}{p}(\rho-1)+\frac{1}{p^2}(\rho^2- 2\rho\cos\theta+ 1)\\
&\leq 1+\frac{2}{p}(\rho-1)+\frac{1}{p^2}((\rho-1)^2+\theta^2)\leq (1+\frac{\rho-1}{p})^2+\frac{\pi^2}{p^2L^2}< 1.
\end{align*}

When $|c_0|=1$ and $|c_i|\le 1$, we prove the result slightly differently. Write $z=\rho z'\in \rho\gamma$ with the same  $\rho= 1-\frac{7}{pL^2}$. Then $f(z)= z^{m_0} f'(z')$, where $f'(z')= c_0+\rho c_1z'+\cdots +\rho^ic_i (z')^i+\cdots$ with $|\rho^ic_i|\le 1$. Apply Corollary 3.2 of Borwein and Erd\'{e}lyi \cite{borwein1997littlewood} to get $|f'(z')|\ge e^{-O(L)}$ for some $z'\in \gamma(L)$. Fixing $z'$ and taking $L= m^{1/3}$, we have
\[|f(z)|= \rho^{m_0}|f'(z')|\ge (1-\frac{7}{pL^2})^m\cdot \exp(-O(L))\ge \exp(-O(m^{1/3})).\]
The corresponding $w=\frac{z-q}{p}$ satisfies  $|w|\le 1$  by (\ref{eq-estimation-w-2}). The proof is completed.
\end{proof}
\vspace{0.3cm}

Before closing this section, we introduce the so-called $W$-contiguous generating function. For any distinct hypermatrices $X,Y\in \{0,1\}^{n^{\times d}}$ and any hypermatrix $W\in \{0,1\}^{l^{\times d}}$ with $l\le n$, we define the \emph{$W$-contiguous generating function of $X-Y$} by
\begin{equation}\label{con-gen-fun}
  h(z_1,\ldots, z_d)= \sum_{k}[1_{X_{k+[l]^d}=W}-1_{Y_{k+[l]^d}=W}]z_1^{k_1}\cdots z_d^{k_d},
\end{equation}
where $k= (k_1,\ldots, k_d)$ runs through $[n]^d$.
Note that $W$ is of the same dimension $d$ as $X,Y$ but of a shorter length. When $d=1$, the function $h$ reduces to the function $B'$ in Eq. (\ref{consebound}).

\section{Dimension reduction}

Recall that Chase \cite{chase2021separating} applied the $W$-generating identity to distinguish sequences $x$ and $y$ which coincide each other on a prefix of length  $\Theta (n^{1/5})$, that is,  $\Theta (n^{1/5})$ contiguous zeros at the beginning of $a=x-y$, by establishing the existence a powerful $W$-generating function; for the other case, that is, $x$ does not look like $y$ at the beginning,  they can be easily distinguished by a trick of Peres and Zhai \cite{peres2017average}. It is a natural thought to apply this case-by-case discussion to hypermatrices. However, the criterion for a similar classification becomes complicated.


To make an effective classification, we design a dimension reduction procedure on $A=X-Y\in \{0,\pm1\}^{n^{\times d}}$, to measure the similarities between $X$ and $Y$. Roughly speaking, we reduce $A$ to a sub-hypermatrix with one dimension lower, one can view it as a slice of $A$,  which separates the thinnest all-zero part of $A$ from the other part. Repeating this procedure on the sub-hypermatrix iteratively until the dimension reduces to zero. Then the thicknesses of the all-zero parts ($\lambda_i$'s below) in this reduction will be our criterion for a classification. This procedure will provide a nice structure of all related generating functions, and enable us to analyse bounds of those functions case-by-case. We give a formal definition of dimension reductions below.

%

%

For any hypermatrices $A\in \{0,\pm 1\}^{n^{\times d}}$,
let $f$ be the generating function of $A$, that is $f(z_1,z_2,\ldots,z_d)=\sum_{k} A_{k} z_1^{k_1}\cdots z_d^{k_d}$, where $k=(k_1,\ldots,k_d)$ runs over $[n]^d$. Now, we define a sequence of hypermatrices $A^d=A, A^{d-1}, \ldots, A^1, A^0=\pm 1$, associated with a sequence of integers $\lambda_d\leq \lambda_{d-1}\leq \ldots\leq \lambda_1$ as follows. Each  $A^i$ will be a nonzero sub-hypermatrix of $A$ on some position $(k_1,\ldots,k_d)$ in $\I(n^{\times d},n^{\times i})$ for $0\leq i\leq d$, but we index its entry position only using the first $i$ parameters $(k_1,\ldots,k_i)$. So in fact, $A^i$ will be of dimension $i$. 

 For each $i\geq 1$, we define the reduction from $A^{i}$ to $A^{i-1}$ iteratively. Suppose we have obtained a sequence of nonzero hypermatrices $A^d=A, A^{d-1}, \ldots, A^i$, and a sequence of non-negative integers $\lambda_d\leq \lambda_{d-1}\leq \ldots\leq \lambda_{i+1}$.  Let $\lambda_i$ be the least integer such that 
\begin{center}
some entry $A^i_{k_1,\ldots, k_i}\neq 0$ with $\min\{k_1,\ldots, k_i\}= \lambda_i$ or $\max\{k_1,\ldots, k_i\}= n-1-\lambda_i$.
\end{center}
Clearly $\lambda_i\le \lfloor\frac{n-1}{2}\rfloor$.
After a permutation on the set of dimensions and possibly an element reversion along some dimension, we have 
\begin{center}
any entry $A^i_{k_1,\ldots, k_i}= 0$ with $k_i< \lambda_{i}$; and some entry $A^i_{k_1,\ldots, k_i}\neq 0$ with $k_i= \lambda_{i}$.
\end{center}
Then let $A^{i-1}$ be a sub-hypermatrix of $A^i$ by restricting $k_i= \lambda_i$, that is a slice of $A^i$. Clearly $A^{i-1}\neq O$. Let $f_i(z_1,\ldots, z_i)$ be the generating function of $A^i$, then it can be written as
\begin{equation}\label{f-gener}
  f_i(z_1,\ldots, z_i)= z_i^{\lambda_i}(f_{i-1}(z_1,\ldots, z_{i-1})+ z_i\cdot *),
\end{equation}
where $f_{i-1}$ is the generating function of  $A^{i-1}$ and the symbol * denotes an uncertain polynomial.

After $d$ steps, we succeed to make a dimension reduction $\R: A\to A^{d-1}\to \cdots \to A^1\to A^0= \pm1$, with a sequence of integers $0\le \lambda_d\le \lambda_{d-1}\le \cdots \le \lambda_1\le \lfloor\frac{n-1}{2}\rfloor$, such that each $A^{i-1}$ is a slice of $A^i$ whose small side contains only zero entries. 
If $A=X-Y$, then we say the two hypermatrices $X$ and $Y$ are \emph{$(\lambda_1,\lambda_2,\ldots, \lambda_d)$-identical} with respect to the dimension reduction $\R: A=X-Y\to A^{d-1}\to \cdots \to A^1\to \pm1$. Note that in the case that $\lambda_{r+1}< l\leq \lambda_{r}$ for some $r\geq 1$ and some proper threshold $l$, it roughly means that on some $r$ dimensions $X$ looks like $Y$, but on the other dimensions, $X$ does not look like $Y$.

\vspace{0.3cm}

Before giving a construction of the required $W$ in Lemma~\ref{lem-reduceto-Wgen-function}, we introduce the conception of periodic hypermatrices.
For any positive integers  $s< l$, we say a hypermatrix $W\in \{0,1\}^{l^{\times d}}$ is \emph{$s$-periodic}, if $W$ has a nonzero period $t\in [-s, s]^d$, that is, for any  $k\in [l]^d$ with $k+t\in [l]^d$, we have $W_k=W_{k+t}$. Otherwise, we say $W$ is \emph{$s$-aperiodic}. 



\begin{lemma}\label{lem-construction-Wgen-function}
Suppose that distinct hypermatrices $X,Y\in \{0,1\}^{n^{\times d}}$ are $(\lambda_1,\lambda_2,\ldots, \lambda_d)$-identical with respect to a dimension reduction $\R: A^d=X-Y\to A^{d-1}\to \cdots \to A^1\to \pm1$. If $\lambda_r\ge l$ for some $r\ge 1$ and an odd number $l\le \lfloor\frac{n-1}{2}\rfloor$, then there exists a hypermatrix $W\in \{0,1\}^{l^{\times r}}$, such that the followings hold.
\begin{enumerate}[(a)]
  \item For each  $r\le i\le d$, let $g_i$ be the $W$-generating function of $A^i$. Then
\begin{equation}\label{g-gener}
    g_{i+1}(z_1,\ldots, z_{i+1})=z_{i+1}^{\lambda_{i+1}}(g_{i}(z_1,\ldots, z_i)+ z_{i+1}\cdot *), ~r\le i\le d-1,
\end{equation}
  where $*$ denotes uncertain polynomials;
  \item Let $h$ be the $W$-contiguous generating function of $A^r$. Then $h$ is nonzero and $\lfloor\frac{l-1}{4}\rfloor$-sparse.
\end{enumerate}
\end{lemma}

\begin{proof}
 Property (a) trivially holds for any $W$ by the definitions of $\lambda_i$.

For Property (b), let $A^r= X^r-Y^r$, where $X^r$ and $Y^r$ are the corresponding $r$-dimensional sub-hypermatrices of $X$ and $Y$, respectively. We will find the hypermatrix $W$ in $X^r$ or $Y^r$. Let $H=\{j\in [n]^r:A^r_j\neq 0\}$ be the collection of positions with nonzero entries in $A^r$. View $H$ as a point set in the Euclidean space $E^r$. Then  there is an $(r-1)$-dimensional hyperplane $\alpha$ such that $\alpha\cap H=j$ for some $j\in H$ and the whole set $H$ lies on exactly one side of $\alpha$. Geometrically, the hyperplane $\alpha$ is tangent to $H$. Since $\lambda_r\ge l$, we have $j+[-\frac{l-1}{2},\frac{l-1}{2}]^r:=\jmath\subset [n]^r$.  Let $W_1=X^r_{\jmath}$ and $W_2=Y^r_{\jmath}$. Then $W_1\neq W_2$ since $X^r_j\neq Y^r_j$. So if we take $W\in \{W_1,W_2\}$, then the $W$-contiguous generating function
\[h(z_1,\ldots, z_r)= \sum_{k\in [n]^r}[1_{X^r_{k+[l]^r}=W}-1_{Y^r_{k+[l]^r}=W}]z_1^{k_1}z_2^{k_2}\cdots z_r^{k_r},\]
is nonzero.

It is left to show the sparsity of $h$. Denote $s=\lfloor\frac{l-1}{4}\rfloor$. First, we claim that if $W$ is $s$-aperiodic, then $h$ is $s$-sparse. Otherwise, suppose that there exist distinct
$k',k''\in [n]^r$ with $|k'_i-k''_i|< s$ for all $1\le i\le r$, such that, without loss of generality, $X^r_{k'+[l]^r}= X^r_{k''+[l]^r}= W$. Then for any $k\in [l]^r$ with $k+(k''-k')\in [l]^r$, $W_k=X^r_{k''+k}=X^r_{k'+k+(k''-k')}=W_{k+(k''-k')}$.  This implies that $W$ has a period $(k''-k')\in [-s, s]^r$, contradicting to the aperiodicity of $W$.

By the claim, we need to show that one of $W_1$ and $W_2$ is $s$-aperiodic. Suppose that both $W_1$ and $W_2$ are $s$-periodic, and let $p_1,p_2\in [-s, s]^r$ be their periods, respectively. Then $j+p_1+p_2\in \jmath$. Choose suitable $p_1$ from $\{\pm p_1\}$, $p_2$ from $\{\pm p_2\}$, such that $p_1+p_2\neq 0$ and they point to the same side of $\alpha$ against $H$. On this side, corresponding entries in $X^r$ and $Y^r$ are equal except  $X^r_j\neq Y^r_j$. However, combining the periodicity of $W_1$ and $W_2$, we have
\[X^r_j=X^r_{j+p_1}=Y^r_{j+p_1}=Y^r_{j+p_1+p_2}=X^r_{j+p_1+p_2}=X^r_{j+p_2}=Y^r_{j+p_2}=Y^r_j,\]
a contradiction. So we complete the proof.
\end{proof}

\vspace{0.3cm}

In what follows, under the notation of a dimension reduction $\R: A\to A^{d-1}\to \cdots \to A^1\to \pm1$, we always denote $f_{i}$  the generating function of  $A^{i}$,   $g_{i}$  the $W$-generating function of  $A^{i}$ for some related $W$, and  $*$ an uncertain polynomial.

\section{Trace reconstruction of matrices and cubes}

In this section, we will give upper bounds for trace reconstruction of matrices and cubes, proving Theorem~\ref{theorem-1} (1)-(2). For matrices,
 we only need to show that for any matrix pair $(X,Y)$, we can find the existence of a matrix $W$ so that the  $W$-generating function of $X-Y$ satisfies the conditions in Lemma~\ref{lem-reduceto-Wgen-function} with $m=n^{3/7}\log^{10/3}n$.

\begin{lemma}\label{lem-function-lowerbound-matrix}
For any distinct matrices $X,Y\in \{0,1\}^{n\times n}$, there exists some $W\in \{0,1\}^{l^{\times r}}$ with $l=4n^{1/7}+1$ and $ 0\le r\le 1$,
such that the $W$-generating function $g(z_1, z_2)$ of $X-Y$ satisfies
\[|g(z_1,z_2)|\ge \exp(-C n^{3/7}\log^{10/3}n) \text{ for some $w=(w_1,w_2)\in \bbC(2,l^{\times r})$ with } \ell(w)\le \exp({C}{n^{-4/7}}),\]
where $C> 0$ is a constant. 
\end{lemma}

\begin{proof}
Suppose that $X$ and $Y$ are $(\lambda_1,\lambda_2)$-identical with respect to a dimension reduction $\R: A=X-Y\to A^1\to \pm1$. 

\textbf{Case 1}. $\lambda_1< l$.

Set $r=0$ and $W=1$. Let $f(z_1,z_2)$ be the generating function of $A$. It is sufficient to show that
\[|f(z_1,z_2)|\ge e^{-O(l^{5/9})} \text{ for some $w_1, w_2\in \bbC$ with } |w_1|,|w_2|\le 1.\]
By Eq. (\ref{f-gener}),
\[f(z_1,z_2)=z_2^{\lambda_2}(f_1(z_1)+z_2\cdot *), ~~f_1(z_1)=z_1^{\lambda_1}(\pm1+z_1\cdot *).\]
Applying Lemma~\ref{lem-polynomial-lowerbound} for $m=l$ and $|c_0|=1,|c_i|\le 1$, we have
\begin{center}
$|f_1(z_1)|\ge e^{-O(l^{1/3})}$ for some $w_1\in \bbC$ with $|w_1|\le 1$.
\end{center}
Fix $w_1$ and $z_1$,  and let $f'(z_2):=f(z_1,z_2)$. We would like to apply Lemma~\ref{lem-polynomial-lowerbound} to $f'(z_2)$ again. Since $m_0=\lambda_2\leq \lambda_1<l$, $|c_0|=|f_1(z_1)|\ge e^{-O(l^{1/3})}$, and $\sum_i|c_i|$ is at most the number of nonzero terms in $A$, which is upper bounded by $ n^2 \le e^{O(l^{1/3})}$, then by Lemma~\ref{lem-polynomial-lowerbound} with $m=l$ and $a={1/3}$, we have
\begin{center}
$|f(z_1,z_2)|=|f'(z_2)|\ge e^{-O(l^{5/9})}$ for some $w_2\in \bbC$ with $|w_2|\le 1$.
\end{center}

\textbf{Case 2}. $\lambda_1\ge l$. 

Set $r=1$ and take the sequence $W\in \{0,1\}^l$ in Lemma~\ref{lem-construction-Wgen-function}.  Let $g(z_1,z_2)$ be the $W$-generating function of $A$, where $(z_1,z_2)\in \bbC(2,l^{\times 1})$.
By Eq. (\ref{g-gener}),
\[g(z_1,z_2)=z_2^{\lambda_2}(g_1(z_1)+z_2\cdot *),\]
where $g_1$ is the $W$-generating function of the sequence $A^1$. 
By Lemma~\ref{lem-construction-Wgen-function} (b), the $W$-contiguous generating function of $A^1$, say $h(z)$, is nonzero and $n^{1/7}$-sparse. Let $d_0$ be the minimum degree of $h$. Then we can write
\[h(z)= \pm z^{d_0}(1- \epsilon z^{d_1}+ \sum_{j=n^{1/7}}^n c_jz^j):= \pm z^{d_0}h_0(z)\]
with $1\le d_1< n^{1/7}, \epsilon\in \{0,1\}$ and $|c_j|\le 1$ for each $j$. Clearly, $h_0(z)$ satisfies the condition in Theorem~\ref{thm-sparsepolynomial-lowerbound} for $\mu=1/7$. Hence,
$|h_0(e^{i\theta})|\ge \exp(-O(n^{1/7} \log^5n))$ for some $|\theta|\le n^{-2/7}$. Let  $z_0=e^{i\theta}\in \gamma(\pi n^{2/7})\subset \gamma(n^{2/7})$. Then $|h(z_0)|= |h_0(z_0)|\ge \exp(-O(n^{1/7} \log^5n))$.
Applying Lemma~\ref{lem-sequence-contiguous-function}, we get
\[|g_1(z_1)|\ge e^{-O(n^{1/7}\log^5 n)}\]
for some  $z_1= (z_{10},\ldots, z_{1,l-1})$ with $z_{10}=z_0$ and $z_{1j}\in [1-2p,1]$, $j\geq 1$. By Eq. (\ref{eq-estimation-w}), the corresponding $w_1\in \bbC^l$ satisfies $\ell(w_1)\le e^{O(n^{-4/7})}$.

Fixing $w_1$ and $z_1$,  let $g'(z_2)=g(z_1,z_2)$ which is single-variable. We apply Lemma~\ref{lem-polynomial-lowerbound} to $g'(z_2)$. Since $m_0=\lambda_2< n$, $|c_0|=|g_1(z_1)|\ge e^{-O(n^{1/7}\log^5 n)}$, and $\sum_i|c_i|$ is at most the number of nonzero terms in the Littlewood polynomial $g(z_1,z_2)$, which is upper bounded by $n\binom{n}{l} \le e^{O(l\log n)}\le e^{O(n^{1/7}\log^5 n)}$, then by Lemma~\ref{lem-polynomial-lowerbound} with $m=n$ and $a=\frac{1}{7}+ 5\frac{\log \log n}{\log n}$, we have
\begin{center}
$|g(z_1,z_2)|= |g'(z_2)|\ge e^{-O(n^{3/7}\log^{10/3}n)}$ for some $w_2\in \bbC$ with $|w_2|\le 1$.
\end{center}
\end{proof}
\vspace{5pt}

Combining Lemma~\ref{lem-reduceto-Wgen-function} and Lemma~\ref{lem-function-lowerbound-matrix}, Theorem~\ref{theorem-1} (1) immediately follows.

As for trace reconstruction of cubes, we can make the same classification as for matrices, resulting the following lemma, whose proof is moved to Appendix. Hence Theorem~\ref{theorem-1} (2) follows by applying Lemma~\ref{lem-reduceto-Wgen-function} with $m=n^{5/9}\log^{5/2}n$. 

\begin{lemma}\label{lem-function-lowerbound-cube}
For any distinct cubes $X,Y\in \{0,1\}^{n\times n\times n}$, there exists some $W\in \{0,1\}^{l^{\times r}}$ with $l=4n^{1/9}+1$ and $ 0\le r\le 1$,
 such that the $W$-generating function $g(z_1, z_2, z_3)$ of $X-Y$ satisfies
\[|g(z)|\ge \exp(-C n^{5/9}\log^{5/2}n) \text{ for some $w=(w_1, w_2, w_3)\in \bbC(3,l^{\times r})$ with } \ell(w)\le \exp(\frac{C\log^{5/2}n}{n^{4/9}}),\]
where $C> 0$ is a constant.
\end{lemma}

Following the proof ideas for matrices and cubes, one may get an upper bound $\exp(\widetilde{O}(n^{1-\frac{4}{2d+3}}))$ for hypermatrices of general dimension $d\ge 4$, which is slightly better than the bound $\exp(\widetilde{O}(n^{d/(d+2)}))$ of Krishnamurthy et al. \cite{krishnamurthy2021trace}.
 However, this is still close to the trivial bound $\exp(O(n))$ for high dimensions unfortunately. To settle  this problem, we establish a Littlewood-type result for multi-variable polynomials with certain sparse property, that  is Theorem~\ref{thm-sparsepolynomial-multivariate-lowerbound}, which generalizes the single-variable version in  Theorem~\ref{thm-sparsepolynomial-lowerbound}.  Such a result on lower bounding sparse polynomials at some restricted point is useful, since certain $W$-contiguous generating function is sparse  by Lemma~\ref{lem-construction-Wgen-function} (b).
 See details in the next section.

\section{Trace reconstruction of hypermatrices}
In this section, we always assume $d\ge 4$ is fixed. First, we recall Theorem~\ref{thm-sparsepolynomial-multivariate-lowerbound} and give a proof of it. 
In this proof, we apply a technical method to reduce the sparse multi-variable polynomial to one-variable. We realize this by finding a suitable hyperplane tangent to a sparse set of points with at most two tangent points in the Euclidean space, based on the ideas in \cite{zhang2024reconstruction}. This is different from the proof of  Theorem~\ref{thm-sparsepolynomial-lowerbound} for the one-variable case by using complex analysis.

%
\begin{thmbis}{thm-sparsepolynomial-multivariate-lowerbound}
Let $h(z_1,\ldots, z_d)= \sum_{k\in [n]^d} c_kz_1^{k_1}\cdots z_d^{k_d}$ be a nonzero $n^\mu$-sparse polynomial with $\mu\in [0,1)$ and each $c_k\in \{0,\pm 1\}$. Then for any $\Delta\ge 1$ and any $L$ with $1\le L\le n^\Delta$,
\[\max_{z_1,\ldots, z_d\in \gamma(L)}|h|\ge e^{-O(\Delta L n^{1-\mu} \log n)}.\]
\end{thmbis}
\begin{proof}
The idea of this proof is to reduce the multi-variable polynomial $h$ to a single-variable polynomial, for which we can find a big positive value at some restricted point by the similar technique as in the proof of Lemma~\ref{lem-polynomial-lowerbound}. The way of reduction is to set $z_i=u^{b_i}$ for some $b= (b_1,\ldots, b_d)\in \bbZ^d$, then $h$ becomes $\sum_k c_ku^{b\cdot k}$, a polynomial say $\hbar$ in $u$. By this way, the coefficient of a monomial  $u^m$ in  $\hbar$ is the sum of $c_k$ over all $k$'s satisfying $b\cdot k=m$, which are locating on a hyperplane. However, a randomly chosen $b$ may result in all coefficients in $\hbar$ being zero, thus $\hbar$ is a zero function, and useless. For this reason, we need the integer vector $b$ satisfying the following property,

(P): If $m_0$ is the smallest integer such that $b\cdot k=m_0$ for some $k\in H:= \{k\in [n]^d: c_k\neq 0\}$, then the number of such $k$'s in $H$ is at most two. This could be guaranteed by the sparsity of $h$.

Note that $m_0$ might be the possible minimum degree of nonzero monomials in $\hbar$. If there is only one $k\in H$ satisfying  $b\cdot k=m_0$, then the coefficient of $u^{m_0}$ is $c_k\neq 0$, thus $m_0$ is the  minimum degree and $\hbar$ is nonzero. However, if there are two such $k$'s, say $k^1$, $k^2$, then the coefficient is $c_{k^1}+c_{k^2}$ which might be zero if $c_{k^1}=-c_{k^2}$. In this case, we will introduce a unit root number $v$ and set $z_i=u^{b_i}v$ for some index $i$ to make sure the coefficient of  $u^{m_0}$ nonzero.

Further, since we aim to lower bound $h$ on some point with each $z_i$ close to $1$ on the unit circle, or equivalently, we aim to lower bound $\hbar$ on some point $u\in \gamma$ with  $u^{b_i}$ close to $1$, then each integer $b_i$ should be small, or the vector $b$ should be of a short length.

%
%

Now we formally construct $b$, which will be chosen from the following set
\[\N(R)= \{(b_1,\ldots, b_d)\in \bbZ^d: \gcd(b_1,\ldots, b_d)=1, \sqrt{b_1^2+\cdots +b_d^2}\le R\},\] where $R=dn^{1-\mu}$.
View $H$ as a point set in the $d$-dimensional Euclidean space $E^d$.
Take a sphere shell $S$ of radius $\frac{\sqrt{d}}{2}n$ enclosing $H$, and a circumscribed hyperpolygon $P$ of $S$ that is formed by all $(d-1)$-dimensional hyperplanes with normal vectors in $\N(R)$. By a geometric argument, we obtain a claim below, whose  proof is moved to Appendix.
\vspace{-0.1cm}
\begin{claim}\label{claim-diameter-small}
For $n$ sufficiently large, each facet of $P$ has a diameter less than $n^\mu$.
\end{claim}

Translate $P$ until $P$ touches $H$ for the first time, that is, $P\cap H\neq\emptyset$ and $H$ is still inside $P$. Let $\beta$ be a facet of $P$ touching $H$, and $b=(b_1,\ldots,b_d)\in \N(R)$ be the normal vector of $\beta$. Then $H$ lies on one side of $\beta$. Choose $b\in \{\pm b\}$ such that $b$ points to the $H$-side with respect to $\beta$. Then $\beta$ belongs to the hyperplane $b\cdot x= m_0$, and $m_0=\min \{b\cdot k: k\in H\}$.

By the claim, $\beta$ has a diameter less than $n^\mu$.
Let $H_1= \{k\in H: c_k=1\}$ and $H_2= \{k\in H: c_k=-1\}$ form a bipartition of $H$. For each $i=1,2$, the Euclidean distance of any two distinct points  in $H_i$ is at least $n^\mu$ by the $n^\mu$-sparsity of $h$.
Hence $|\beta\cap H_1|\leq 1$, $|\beta\cap H_2|\leq 1$, and thus $|\beta\cap H|\leq 2$ satisfying property (P).


Given $b$ above, we are able to take a variable substitution that $z_1=u^{b_1}v, z_2=u^{b_2},\ldots, z_d=u^{b_d}$ where $v=e^{\frac{\pi i}{2n^\Delta}}$ is a fixed number. 
Now we write
\[h(z_1,\ldots, z_d)= \sum_{k\in [n]^d} c_kz_1^{k_1}\cdots z_d^{k_d}= \sum_{k\in [n]^d} c_k v^{k_1}u^{b\cdot k}=:\hbar(u).\]
Then $\hbar$ has possible minimum degree $m_0$, so we write
\[\hbar(u):= u^{m_0}(C_v+ u\cdot *)=: u^{m_0}\hbar_0(u).\]
Note that $C_v$ is the coefficient of $u^{m_0}$, which equals $\sum_{k\in \beta\cap H} c_k v^{k_1}$.  When $|\beta\cap H|= 1$, since $v\in \gamma$, we have $|C_v|=|\hbar_0(0)|= 1$. When $|\beta\cap H|= 2$, then $|\beta\cap H_1|=|\beta\cap H_2|=1$. Denote $k^i=\beta\cap H_i$, $i=1,2$. Without loss of generality, we assume the first coordinates $k^1_1\neq k^2_1$, and assume $k^1_1> k^2_1$. (For other cases, say $k^1_i\neq k^2_i$, we just move $v$ to $z_i$ in the variable substitution step). Since $\beta$ has a diameter less than $n^\mu$, we have $1\leq t:=k^1_1- k^2_1< n^\mu$. Then
\[C_v=c_{k^1}v^{k^1_1}+c_{k^2}v^{k^2_1}=\pm v^{k^2_1}(1-v^t).\]
So \[|\hbar_0(0)|= |C_v|=|1- v^t|= |1- e^{\frac{t\pi i}{2n^\Delta}}|\ge |1- e^{\frac{\pi i}{2n^\Delta}}|\ge e^{-O(\Delta \log n)},\] where the first inequality follows from $1\leq t< n^\mu\le n^\Delta$.


Next, for any $L$ with $1\le L\le n^\Delta$, we will lower bound $|\hbar(u)|$, and consequently lower bound $|h|$ when each $z_i$ belongs to the arc $\gamma(L)$. Define a polynomial
\[F(u)=\prod_{1\leq a\leq L'}\hbar_0(ue^{2\pi i\frac{a}{L'}})\] for some integer $L'$.
Then $|F(0)|=|\hbar_0(0)|^{L'}\ge e^{-O(\Delta L'\log n)}$.
By the maximum modulus principle, $|F(u)|\ge |F(0)|\ge e^{-O(\Delta L'\log n)}$ for some $u\in \gamma$. We may choose $a$ such that $u e^{2\pi i\frac{a}{L'}}\in \gamma(L')$, which is still denoted by $u$. Then
\[e^{-O(\Delta L'\log n)}\le |F(u)|\le |\hbar_0(u)| n^{d(L'-1)}\le |\hbar_0(u)|e^{dL' \log n},\]
where we use the fact that $|\hbar_0(u)|\le n^d$. Taking $L'=2dL n^{1-\mu}$, we have
\[|h(z_1,\ldots, z_d)|=|\hbar(u)|=|\hbar_0(u)|\ge e^{-O(\Delta Ln^{1-\mu}\log n)}\]
for some $z_1=u^{b_1}v, z_2=u^{b_2},\ldots, z_d=u^{b_d}$ with $u\in \gamma(2dL n^{1-\mu}), v=e^{\frac{\pi i}{2n^\Delta}}\in \gamma(2n^\Delta)$ and $|b_i|\le dn^{1-\mu}$ for $1\le i\le d$. It is left to show that $z_1,\ldots, z_d\in \gamma(L)$. Indeed,
\begin{align*}
z_i=& u^{b_i}\in \gamma(\frac{2dL n^{1-\mu}}{|b_i|})\subseteq \gamma(2L)\subseteq \gamma(L),  ~~2\le i\le d.
\end{align*}For $z_1$, since $u^{b_1}\in \gamma(2L)$, $v\in \gamma(2n^\Delta)$ and $L\le n^\Delta$, then
\begin{align*}
z_1=& u^{b_1}v\in \gamma(L)\cup \gamma(n^\Delta)= \gamma(L).
\end{align*}
This completes the proof.

\end{proof}

\vspace{0.3cm}

Note that any $d$-variable polynomial of the form $\sum_{k\in [n]^d} c_kz_1^{k_1}\cdots z_d^{k_d}$  is trivially $1$-sparse, which is the case of  $\mu= 0$ in Theorem~\ref{thm-sparsepolynomial-multivariate-lowerbound}. However, by Claim~\ref{claim-diameter-small} saying that each facet $\beta$ of the hyperpolygon $P$ has a diameter less than $1$, $|\beta\cap H|= 1$ must happen. So the parameter $v$ becomes unnecessary, and the variable substitution and the remaining arguments can be simplified. This is  different from the case of $\mu> 1$, and leads to the following corollary. When $L\geq C n^{\frac{1}{d-1}}$ for a large constant  $C$, Corollary~\ref{cor-polynomial-multivariate-lowerbound} greatly improves the lower bound $e^{-O(L^d \log n)}$ in \cite[Lemma 22]{krishnamurthy2021trace} by Krishnamurthy et al.

\begin{corollary}\label{cor-polynomial-multivariate-lowerbound}
Let $h(z_1,\ldots, z_d)= \sum_{k\in [n]^d} c_kz_1^{k_1}\cdots z_d^{k_d}$ be a nonzero polynomial with each $c_k\in \{0,\pm 1\}$. Then
\[\max_{z_1,\ldots, z_d\in \gamma(L)}|h|\ge e^{-O(Ln \log n)}.\]
\end{corollary}

Next, we prove the reduction from  lower bounding a $W$-generating function to lower bounding the corresponding $W$-contiguous generating function, where $W$ is a sub-hypermatrix of full dimensions. This is a multivariate version of Lemma~\ref{lem-sequence-contiguous-function}.

\begin{lemma}\label{lem-hypermatrix-contiguous-function}
For any distinct hypermatrices $X,Y\in \{0,1\}^{n^{\times r}}$ and a hypermatrix $W\in \{0,1\}^{l^{\times r}}$ with $l\le n$, let $g(z_1,\ldots, z_r)$ be the $W$-generating function of $X-Y$ with $z_i=(z_{i0},\ldots, z_{i,l-1})\in \bbC^l, 1\le i\le r$, and let $h(z_{10},\ldots, z_{r0})$ be the $W$-contiguous generating function of $X-Y$. If
\begin{center}
$|h|\ge e^{-O(l\log n)}$ for some $z_{i0}\in \gamma, 1\le i\le r$,
\end{center}
then fixing $\{z_{i0}\}$, there exist $z_{ij}\in [1-2p, 1]$ for $1\le i\le r, 1\le j\le l-1$ such that
\[|g|\ge e^{-O(l\log n)}.\]
\end{lemma}

\begin{proof}
For $k=(k_1,\ldots,k_{r})\in [n]^{r-1}\times \I(n,l)$ with $k_r=(k_{r0},\ldots,k_{r,l-1})\in \I(n,l)$, let $\bar{k}$ denote the position of an $l^{\times r}$-sub-hypermatrix that is contiguous along the first $r-1$ dimensions starting at $(k_1,\ldots,k_{r-1})$ and $k_r$ on the last dimension. Consider the partial sum in $g$ over such positions, denoted by
\[\bar{g}(z_1,\ldots, z_r)=\sum_k [1_{X_{\bar{k}}=W}-1_{Y_{\bar{k}}=W}]z_{10}^{k_1}\cdots z_{r-1,0}^{k_{r-1}}\cdot z_r^{\odot k_r},\]where $k$ runs over $[n]^{r-1}\times \I(n,l)$ and $z_i=(z_{i0},\ldots, z_{i,l-1})\in \bbC^l, 1\le i\le r$. 
Fix $z_{i0}\in \gamma, 1\le i\le r$ such that $|h|\ge e^{-O(l\log n)}$. We claim that there exist $z_{rj}\in [1-2p, 1]$ for $1\le j\le l-1$ such that
$|\bar{g}|\ge e^{-O(l\log n)},$ that is, the norm of the partial sum in the $W$-generating function over sub-hypermatrices whose position is contiguous along the first $r-1$ dimensions can also be lower bounded by $e^{-O(l\log n)}$.
Then using the similar argument, we could remove the contiguous conditions of the rest $r-1$ dimensions one by one. Since $r$ is a constant, we finally obtain
$|g|\ge e^{-O(l\log n)},$
for some $z_{ij}\in [1-2p, 1], 1\le i\le r, 1\le j\le l-1$.

Now we prove the claim. Consider the case when $z_{r1}=z_{r2}=\ldots=z_{r,l-1}$. Then $\bar{g}$ becomes a single-variable polynomial
\[\bar{g}(z_{r1})=\sum_k [1_{X_{\bar{k}}=W}-1_{Y_{\bar{k}}=W}]z_{10}^{k_1}\cdots z_{r-1,0}^{k_{r-1}}\cdot z_{r0}^{k_{r0}}z_{r1}^{k_{r,l-1}-k_{r0}-(l-1)}.\]
Note that when $z_{r1}=0$, we have $z_{r1}^{k_{r,l-1}-k_{r0}-(l-1)}= 1$ only when $k_{r,l-1}-k_{r0}-(l-1)= 0$, that is, $k_r=k_{r0}+[l]$. So, $\bar{g}(0)$ keeps the terms corresponding to contiguous sub-hypermatrices.
Thus $\bar{g}(0)= h(z_{10},\ldots, z_{r0})$. Denote
$\bar{h}(z_{r1})= \binom{n}{l}^{-r} \bar{g}(z_{r1}).$
Since each coefficient of $\bar{h}(z_{r1})$ has norm at most $1$, then by \cite[Theorem 5.1]{borwein1999littlewood}, there are absolute constants $c_1,c_2>0$, such that
\begin{align*}
 \max_{z_{r1}\in [1-2p,1]}|\bar{h}(z_{r1})|\ge&  |\bar{h}(0)|^{c_1/(2p)}e^{-c_2/(2p)}\ge
  \left(\binom{n}{l}^{-r} e^{-O(l\log n)}\right)^{c_1/(2p)}e^{-c_2/(2p)}\\
\ge& e^{-O(l\log n)}.
\end{align*}
Taking a $z_{r1}$ realizing this maximum, we have
\[|\bar{g}(z_{r1})|=\binom{n}{l}^r |\bar{h}(z_{r1})|\ge e^{-O(l\log n)}.\]Thus the claim is proved.
%
%
\end{proof}

\vspace{0.3cm}

Now, we are able to give the existence of a hypermatrix $W$ such that the $W$-generating function has a big norm at some restricted point.

\begin{lemma}\label{lem-function-lowerbound-hypermatrix}
For any distinct hypermatrices $X,Y\in \{0,1\}^{n^{\times d}}$, there exists some $W\in \{0,1\}^{l^{\times r}}$ with $l= 4n^{3/5}+1$ and $ 0\le r\le d$,
 such that the $W$-generating function $g$ of $X-Y$ satisfies
\[|g(z)|\ge \exp(-C n^{3/5}\log n) \text{ for some $w\in \bbC(d,l^{\times r})$ with } \ell(w)\le \exp(\frac{C}{n^{2/5}}),\]
where $C> 0$ is a constant.
\end{lemma}

\begin{proof} Let $A^d=X-Y$.
Suppose that $X$ and $Y$ are $(\lambda_1,\lambda_2,\ldots, \lambda_d)$-identical with respect to a dimension reduction $\R: A^d\to A^{d-1}\to \cdots \to A^1\to \pm1$. 

\textbf{Case 1}. $\lambda_2< l$.

Set $r=0$ and $W=1$. Let $f_i$ be the generating function of $A^i$ for $1\le i\le d$. Then
\[f_{i+1}(z_1,\ldots, z_{i+1})= z_{i+1}^{\lambda_{i+1}}(f_i(z_1,\ldots, z_i)+ z_{i+1}\cdot *), ~1\le i\le d-1.\]
Since $f_1(z_1)$ is a Littlewood polynomial, we can apply \cite[Corollary 3.2]{borwein1997littlewood} to get
\[|f_1(z_1)|\ge e^{-O(l)} \text{ for some } z_1\in \gamma(l).\]
Then the corresponding $w_1$ satisfies $|w_1|\le e^{O(l^{-2})}$ by Eq.~(\ref{eq-estimation-w}).
Fix $w_1$ and $z_1$. 
Since $\lambda_2< l$, we can apply Lemma~\ref{lem-polynomial-lowerbound} to $f_2(z_1,z_2)$ with $m=l$ and $a=1$ to get
\begin{center}
$|f_2(z_1,z_2)|\ge e^{-O(l)}$ for some $w_2\in \bbC$ with $|w_2|\le 1$.
\end{center}
Repeatedly applying Lemma~\ref{lem-polynomial-lowerbound} $(d-1)$ times, we finally obtain
\begin{center}
$|f_d|\ge e^{-O(l)}$ for some $\{w_i\}$ with $|w_i|\le 1, 2\le i\le d$ and $|w_1|\le e^{O(l^{-2})}$.
\end{center}

\textbf{Case 2}. $\lambda_{r+1}< l\le \lambda_r$ for some $r\ge 2$. 

Take the hypermatrix $W\in \{0,1\}^{l^{\times r}}$ in Lemma~\ref{lem-construction-Wgen-function} and let $g(z)$ be the $W$-generating function of $X-Y$, where $z=(z_1,\ldots, z_d)\in \bbC(d,l^{\times r})$. We will show that for some $w=(w_1,\ldots, w_d)\in \bbC(d,l^{\times r})$,
\begin{center}
$|g(z)|\ge e^{-O(n^{3/5} \log n)}$  with $\prod_{i=1}^{r}\ell(w_i)\le e^{O(n^{-2/5})}$ and $|w_i|\le 1, r+1\le i\le d$.
\end{center}
Let $g_i$ be the $W$-generating function of the $i$-dimensional sub-hypermatrix $A^i$ for $r\le i\le d-1$. Then we have the same recursions as in Eq. (\ref{g-gener}).
Let $h(z_{10},\ldots, z_{r0})$ be the $W$-contiguous generating function of $A^r$. By Lemma~\ref{lem-construction-Wgen-function} (b), $h$ is nonzero and $n^{3/5}$-sparse. We apply Theorem~\ref{thm-sparsepolynomial-multivariate-lowerbound} for $\mu=3/5,L=n^{1/5},\Delta=1$ to get
\[|h(z_{10},\ldots, z_{r0})|\ge e^{-O(n^{3/5} \log n)} \text{ for some } z_{10},\ldots, z_{r0}\in \gamma(n^{1/5}).\]
Then by Lemma~\ref{lem-hypermatrix-contiguous-function}, for some $z_{10},\ldots, z_{r0}\in \gamma(n^{1/5})$ and $z_{ij}\in [1-2p,1], 1\le i\le r, 1\le j\le l-1$,
\[|g_r(z_1,\ldots, z_r)|\ge e^{-O(n^{3/5} \log n)}.\]
By Eq.~(\ref{eq-estimation-w}), we have $\prod_{i=1}^{r}\ell(w_i)\le e^{O(n^{-2/5})}$ correspondingly.
Fixing $w_1,\ldots, w_r$ and $z_1,\ldots, z_r$, and let $g'(z_{r+1})=g_{r+1}(z_1,\ldots,z_r, z_{r+1})$. We apply Lemma~\ref{lem-polynomial-lowerbound} to $g'(z_{r+1})$. Since $m_0=\lambda_{r+1}< l\le n^{3/5} \log n$, $|c_0|=|g_r|\ge e^{-O(n^{3/5} \log n)}$, and $\sum_i|c_i|\leq n\binom{n}{l}^r \le e^{O(l \log n)}\le e^{O(n^{3/5} \log n)}$,  by Lemma~\ref{lem-polynomial-lowerbound} with $m=n^{3/5} \log n$ and $a=1$, we have
\begin{center}
$|g_{r+1}(z_1,\ldots,z_r, z_{r+1})|= |g'(z_{r+1})|\ge e^{-O(n^{3/5} \log n)}$ for some $w_{r+1}\in \bbC$ with $|w_{r+1}|\le 1$.
\end{center}
 Applying Lemma~\ref{lem-polynomial-lowerbound} $(d-r)$ times, we finally obtain
\begin{center}
$|g|\ge e^{-O(n^{3/5} \log n)}$  with $\prod_{i=1}^{r}\ell(w_i)\le e^{O(n^{-2/5})}$ and $|w_i|\le 1, r+1\le i\le d$.
\end{center}

Combining Case 1 and Case 2, we complete the proof.
\end{proof}

\vspace{0.3cm}

Combining Lemma~\ref{lem-reduceto-Wgen-function} and Lemma~\ref{lem-function-lowerbound-hypermatrix} with $m=n^{3/5}\log n$, Theorem~\ref{theorem-3} (3) immediately follows.

%

\section*{Appendix}

\subsection{Proof of Eq. (\ref{wgeriden})}
\begin{proof}
For $r+1\le i\le d$, from $z_i =pw_i+q$, we have $ pz_i^{k_i}=p(pw_i+q)^{k_i}=\sum_{j_i}\binom{k_i}{j_i}p^{j_i+1}q^{k_i-j_i}\cdot w_i^{j_i}$. For $1\le i\le r$, we have
\begin{align*}
&p^l z_i^{\odot k_i}=p^l z_{i0}^{k_{i0}}z_{i1}^{k_{i1}-k_{i0}-1} \cdots z_{i,l-1}^{k_{i,l-1}-k_{i,l-2}-1}\\
=&p^l (pw_{i0}+q)^{k_{i0}}(pw_{i1}+q)^{k_{i1}-k_{i0}-1} \cdots (pw_{i,l-1}+q)^{k_{i,l-1}-k_{i,l-2}-1}\\
  =&p^l \sum_{j_{i0}}\binom{k_{i0}}{j_{i0}} p^{j_{i0}}q^{k_{i0}-j_{i0}} w_{i0}^{j_{i0}}\cdot
  \{\sum_{j_{i1}> j_{i0}}\binom{k_{i1}-k_{i0}-1}{j_{i1}-j_{i0}-1} p^{j_{i1}-j_{i0}-1}q^{(k_{i1}-k_{i0})-(j_{i1}-j_{i0})} w_{i1}^{j_{i1}-j_{i0}-1}\\
  &\cdot [\cdots (\sum_{j_{i,l-1}> j_{i,l-2}}\binom{k_{i,l-1}-k_{i,l-2}-1}{j_{i,l-1}-j_{i,l-2}-1} p^{j_{i,l-1}-j_{i,l-2}-1}q^{(k_{i,l-1}-k_{i,l-2})-(j_{i,l-1}-j_{i,l-2})} w_{i,l-1}^{j_{i,l-1}-j_{i,l-2}-1})]
  \}\\
  =&\sum_{j_i}\binom{k_{i0}}{j_{i0}} \prod_{h=1}^{l-1}\binom{k_{ih}-k_{i,h-1}-1}{j_{ih}-j_{i,h-1}-1}  p^{j_{i,l-1}+1}q^{k_{i,l-1}-j_{i,l-1}}\cdot w_i^{\odot j_i}.
\end{align*}
Therefore, the left hand side of the $W$-generating identity is
\begin{align*}
  &\sum_j\cdot \sum_k 1_{\widetilde{X}_k=W} [\binom{k_{1,0}}{j_{1,0}} \prod_{h=1}^{l-1}\binom{k_{1,h}-k_{1,h-1}-1}{j_{1,h}-j_{1,h-1}-1}  p^{j_{1,l-1}+1}q^{k_{1,l-1}-j_{1,l-1}}\cdot w_1^{\odot j_1}]\\
  &\cdots [\binom{k_{r,0}}{j_{r,0}} \prod_{h=1}^{l-1}\binom{k_{r,h}-k_{r,h-1}-1}{j_{r,h}-j_{r,h-1}-1}  p^{j_{r,l-1}+1}q^{k_{r,l-1}-j_{r,l-1}}\cdot w_r^{\odot j_r}]\\
  &\cdot [\binom{k_{r+1}}{j_{r+1}}p^{j_{r+1}+1}q^{k_{r+1}-j_{r+1}}\cdot w_{r+1}^{j_{r+1}}]\cdots [\binom{k_d}{j_d}p^{j_d+1}q^{k_d-j_d}\cdot w_d^{j_d}]\\
  =&\sum_k 1_{\widetilde{X}_k=W} \sum_j \cdots = \sum_k 1_{\widetilde{X}_k=W} \sum_{j_0}\sum_{j_1}\cdots \sum_{j_d} \cdots\\
  =&\sum_k 1_{\widetilde{X}_k=W} (p^l z_1^{\odot k_1})\cdots (p^l z_r^{\odot k_r})\cdot (p z_{r+1}^{k_{r+1}})\cdots (p z_d^{k_d})\\
  =&p^{rl +d-r}g(z_1,\ldots, z_d).
\end{align*}
\end{proof}

\subsection{Proof of Lemma~\ref{lem-function-lowerbound-cube}}

\begin{proof}
Suppose that $X$ and $Y$ are $(\lambda_1,\lambda_2,\lambda_3)$-identical with respect to a dimension reduction $\R: A=X-Y\to A^2\to A^1\to \pm1$. 

\textbf{Case 1}. $\lambda_1< l$.

Set $r=0$ and $W=1$. Let $f(z_1,z_2,z_3)$ be the generating function of $A$. Similar to Case 1 of matrices, we apply Lemma~\ref{lem-polynomial-lowerbound} three times to get
\begin{center}
$|f|\ge e^{-O(l^{19/27})}$ for some $w_1,w_2,w_3\in \bbC$ with $|w_1|,|w_2|,|w_3|\le 1$.
\end{center}

\textbf{Case 2}. $\lambda_1\ge l$.

Set $r=1$ and take the sequence $W\in \{0,1\}^l$ in Lemma~\ref{lem-construction-Wgen-function}. Let $g(z_1,z_2,z_3),g_2(z_1,z_2),g_1(z_1)$ be the $W$-generating function of $A,A^2,A^1$, respectively, where $(z_1,z_2,z_3)\in \bbC(3,l^{\times 1})$.
By Eq. (\ref{g-gener}), we write
\[g(z_1,z_2,z_3)=z_3^{\lambda_3}\left(g_2(z_1,z_2)+z_3g'_2(z_1,z_2)\right),\text{ and }g_2(z_1,z_2)=z_2^{\lambda_2}(g_1(z_1)+z_2g'_1(z_1)),\]
where $g'_2$ and $g'_1$ are polynomials. Similar to Case 2 of matrices, we have
\begin{center}
$|g_1|\ge e^{-O(n^{1/9}\log^5 n)}$ for some $w_1\in \bbC^l$ with  $\ell(w_1)\le e^{O(n^{-4/9})}$.
\end{center}
Fix $w_1$ and $z_1$, and define the polynomial for some integer $L$,
\[F(z_2,z_3)=\prod_{1\leq a,b\leq L}g(z_1, z_2e^{2\pi i\frac{a}{L}}, z_3e^{2\pi i\frac{b}{L}}).\]
Further, define
\begin{align*}G(z_2,z_3)&=\prod_{1\leq a,b\leq L}\left(g_2(z_1,z_2e^{2\pi i\frac{a}{L}})+z_3e^{2\pi i\frac{b}{L}}g'_2(z_1,z_2e^{2\pi i\frac{a}{L}})\right),\\
H(z_2)&=\prod_{1\leq a\leq L}\left(g_1(z_1)+z_2e^{2\pi i\frac{a}{L}}g'_1(z_1)\right).
\end{align*}
Then for any $z_2,z_3\in \gamma$, $|G(z_2,z_3)|=|F(z_2,z_3)|$, $|G(z_2,0)|=|H(z_2)|^L$ and $|H(0)|=|g_1(z_1)|^{L}$.
By the maximum modulus principle, there exist some $z_2,z_3\in \gamma$ such that
\[|F(z_2,z_3)|=|G(z_2,z_3)|\geq |G(z_2,0)|=|H(z_2)|^L\geq |H(0)|^L= |g_1(z_1)|^{L^2}\geq \exp(-O(n^{1/9} L^2\log^5 n)).\]
We may choose $a,b$ such that $z_2 e^{2\pi i\frac{a}{L}},z_3 e^{2\pi i\frac{b}{L}}\in \gamma(L)$, which are still denoted by $z_2,z_3$, respectively. Then
\begin{align*}
\exp(-O(n^{1/9} L^2\log^5 n))\leq |F(z_2,z_3)|\leq |g(z_1,z_2,z_3)| \left(n^2\binom{n}{l}\right)^{L^2-1}\leq |g(z_1,z_2,z_3)|\exp(7n^{1/9}L^2\log n),
\end{align*}
where we use the fact that $|g(z_1,z_2,z_3)|\leq n^2\binom{n}{l}$. By Eq.~(\ref{eq-estimation-w}), $|w_2|, |w_3|\le \exp(O(L^{-2}))$. Taking $L=n^{2/9}/ \log^{5/4}n$,
\[|g|\ge \exp(-O(n^{5/9}\log^{5/2}n)) \text{ for some $w=(w_1, w_2, w_3)\in \bbC(3,l^{\times r})$ with } \ell(w)\le \exp(O(\frac{\log^{5/2}n}{n^{4/9}})).\]

Combining Case 1 and Case 2, we complete the proof.
\end{proof}

\subsection{Proof of Claim~\ref{claim-diameter-small}}
\begin{proof}
Let $n$ be sufficiently large, then $R= dn^{1-\mu}$ is large. Assume that the center of $S$ is the origin $\vo$. Take $\vx_0\in P$ maximizing the Euclidean distance from $\vo$. Let $\overline{\vo\vx_0}$ be the segment ending at $\vo$ and $\vx_0$, then take a point $\vx\in \overline{\vo\vx_0}$ such that $\|\vx\|= R-\sqrt{d}$. Without loss of generality, let $\vx= (x_1,\ldots, x_d)$ with $x_1\ge x_2\ge \cdots \ge x_d\ge 0$. Clearly $x_1\ge \frac{R}{\sqrt{d}}-1$.
Take another point $\vy\in \overline{\vo\vx_0}$ such that $\vy= (p, y_2,\ldots, y_d)$ with $p\ge y_2\ge \cdots \ge y_d\ge 0$, where $p$ is the largest prime satisfying $p\le x_1$. Clearly $p= (1-o(1))x_1$ and $\|\vy\|= (1-o(1))(R-\sqrt{d})$.
Take $\va= (p, a_2,\ldots, a_d)\in \bbZ^d$ with $a_i\in [p-1]$ and $|a_i-y_i|\le 1$ for $2\le i\le d-1$. Clearly, $\gcd(p, a_2,\ldots, a_d)= 1$, $\|\va- \vy\|\le \sqrt{d-1}$ and thus $\|\va\|= (1-o(1))R$. Therefore, $\va\in \N(R)$.

Let $\theta$ be the angle formed by $\overline{\vo\va}$ and $\overline{\vo\vy}$. Since $\|\va- \vy\|= o(\|\va\|)$, we have
\[\tan \theta\le (1+o(1))\frac{\|\va- \vy\|}{\|\va\|}\le (1+o(1))\frac{\sqrt{d-1}}{R}.\]
Extend the segment $\overline{\vo\va}$ to intersect $S$ at point $\va_0$, then $\|\va_0\|= \frac{\sqrt{d}}{2}n$.
Since $\va\in \N(R)$, $\va_0$ is a tangent point of $S$ and some facet $\alpha$ of $P$.
Extend $\overline{\vo\vx_0}$ to intersect $\alpha$ at a point $\vy_0$, then $\overline{\va_0\vy_0}$ and $\overline{\va_0\vo}$ span a right angle. On the other hand, by the extreme property of $\vx_0$, the point $\vy_0$ is either the point $\vx_0$ itself or outside $P$.
Therefore, each facet of $P$ has a diameter at most
\[2\|\vy_0-\va_0\|= 2\|\va_0\| \tan\theta \le (1+o(1))\frac{\sqrt{d(d-1)}n}{R}< \frac{dn}{R}= n^\mu.\]
\end{proof}

\end{document}